\documentclass[12pt]{amsart}
\usepackage{amsmath,amssymb,amsbsy,amsfonts,latexsym,amsopn,amstext,
                                               amsxtra,euscript,amscd,bm}
                    
\usepackage{url}
\usepackage[colorlinks,linkcolor=blue,anchorcolor=blue,citecolor=blue]{hyperref}
\usepackage{color}

\begin{document}
\newtheorem{problem}{Problem}
\newtheorem{theorem}{Theorem}
\newtheorem{lemma}[theorem]{Lemma}
\newtheorem{claim}[theorem]{Claim}
\newtheorem{cor}[theorem]{Corollary}
\newtheorem{prop}[theorem]{Proposition}
\newtheorem{definition}{Definition}
\newtheorem{question}[theorem]{Question}

\def\cA{{\mathcal A}}
\def\cB{{\mathcal B}}
\def\cC{{\mathcal C}}
\def\cD{{\mathcal D}}
\def\cE{{\mathcal E}}
\def\cF{{\mathcal F}}
\def\cG{{\mathcal G}}
\def\cH{{\mathcal H}}
\def\cI{{\mathcal I}}
\def\cJ{{\mathcal J}}
\def\cK{{\mathcal K}}
\def\cL{{\mathcal L}}
\def\cM{{\mathcal M}}
\def\cN{{\mathcal N}}
\def\cO{{\mathcal O}}
\def\cP{{\mathcal P}}
\def\cQ{{\mathcal Q}}
\def\cR{{\mathcal R}}
\def\cS{{\mathcal S}}
\def\cT{{\mathcal T}}
\def\cU{{\mathcal U}}
\def\cV{{\mathcal V}}
\def\cW{{\mathcal W}}
\def\cX{{\mathcal X}}
\def\cY{{\mathcal Y}}
\def\cZ{{\mathcal Z}}

\def\A{{\mathbb A}}
\def\B{{\mathbb B}}
\def\C{{\mathbb C}}
\def\D{{\mathbb D}}
\def\E{{\mathbb E}}
\def\F{{\mathbb F}}
\def\G{{\mathbb G}}
\def\I{{\mathbb I}}
\def\J{{\mathbb J}}
\def\K{{\mathbb K}}
\def\L{{\mathbb L}}
\def\M{{\mathbb M}}
\def\N{{\mathbb N}}
\def\O{{\mathbb O}}
\def\P{{\mathbb P}}
\def\Q{{\mathbb Q}}
\def\R{{\mathbb R}}
\def\S{{\mathbb S}}
\def\T{{\mathbb T}}
\def\U{{\mathbb U}}
\def\V{{\mathbb V}}
\def\W{{\mathbb W}}
\def\X{{\mathbb X}}
\def\Y{{\mathbb Y}}
\def\Z{{\mathbb Z}}

\def\ep{{\mathbf{e}}_p}
\def\em{{\mathbf{e}}_m}
\def\eq{{\mathbf{e}}_q}

\def\scr{\scriptstyle}
\def\\{\cr}
\def\({\left(}
\def\){\right)}
\def\[{\left[}
\def\]{\right]}
\def\<{\langle}
\def\>{\rangle}
\def\fl#1{\left\lfloor#1\right\rfloor}
\def\rf#1{\left\lceil#1\right\rceil}
\def\le{\leqslant}
\def\ge{\geqslant}
\def\eps{\varepsilon}
\def\mand{\qquad\mbox{and}\qquad}

\def\sssum{\mathop{\sum\ \sum\ \sum}}
\def\ssum{\mathop{\sum\, \sum}}
\def\ssumw{\mathop{\sum\qquad \sum}}

\def\vec#1{\mathbf{#1}}
\def\inv#1{\overline{#1}}
\def\num#1{\mathrm{num}(#1)}
\def\dist{\mathrm{dist}}

\def\fA{{\mathfrak A}}
\def\fB{{\mathfrak B}}
\def\fC{{\mathfrak C}}
\def\fU{{\mathfrak U}}
\def\fV{{\mathfrak V}}

\newcommand{\bflambda}{{\boldsymbol{\lambda}}}
\newcommand{\bfxi}{{\boldsymbol{\xi}}}
\newcommand{\bfrho}{{\boldsymbol{\rho}}}
\newcommand{\bfnu}{{\boldsymbol{\nu}}}

\def\GL{\mathrm{GL}}
\def\SL{\mathrm{SL}}

\def\Hba{\overline{\cH}_{a,m}}
\def\Hta{\widetilde{\cH}_{a,m}}
\def\Hb1{\overline{\cH}_{m}}
\def\Ht1{\widetilde{\cH}_{m}}

\def\flp#1{{\left\langle#1\right\rangle}_p}
\def\flm#1{{\left\langle#1\right\rangle}_m}
\def\dmod#1#2{\left\|#1\right\|_{#2}}
\def\dmodq#1{\left\|#1\right\|_q}

\def\Zm{\Z/m\Z}

\def\Err{{\mathbf{E}}}

\newcommand{\comm}[1]{\marginpar{%
\vskip-\baselineskip 
\raggedright\footnotesize
\itshape\hrule\smallskip#1\par\smallskip\hrule}}

\def\xxx{\vskip5pt\hrule\vskip5pt}


\title{On the constant in the Polya-Vinogradov inequality}

\author[B. Kerr] {Bryce Kerr}

\address{Department of Pure Mathematics, University of New South Wales,
Sydney, NSW 2052, Australia}
\email{bryce.kerr@unsw.edu.au}

\date{\today}
\pagenumbering{arabic}


 \begin{abstract} 
In this paper we obtain a new constant in the P\'{o}lya-Vinogradov inequality. Our argument follows previously established techniques which use the Fourier expansion of an interval to reduce to Gauss sums. Our improvement comes from approximating an interval by a function with slower decay on the edges and this allows for a better estimate of the $\ell_1$ norm of the Fourier transform. This approximation induces an error for our original sums which we deal with by combining some ideas of Hildebrand with Garaev and Karatsuba concerning long character sums.
 \end{abstract}

\maketitle
\section{Introduction}
Given integers $q,M$ and $N$ and a primitive multiplicative character $\chi$ mod $q$ we consider estimating the sums
$$S(\chi,M,N)=\sum_{M<n \le M+N}\chi(n),$$
and when $M=0$ we write 
$$S(\chi,0,N)=S(\chi,N).$$
The first nontrivial result in this direction is due to P\'{o}lya and Vinogradov from the early 1900's  and states that 
\begin{align}
\label{eq:polyavinogradov}
S(\chi,M,N)\le cq^{1/2}\log{q},
\end{align}
for some constant $c$ independent of $q$. Up to improvements in the constant $c$ this bound has remained sharpest known for the past 100 years and a fundamental question in the area of character sums is whether $c$ can be taken arbitrarily small.  Montgomery and Vaughan~\cite{MV} have shown conditionally on the Generalized Riemann Hypothesis that
\begin{align*}
S(\chi,M,N)\ll q^{1/2}\log\log{q}.
\end{align*}
This would be best possible since Payley~\cite{Pay} has shown that there exists an infinite sequence of integers $q$ and characters $\chi$ mod $q$ such that 
$$\max_{1\le N<q}S(\chi,N) \gg q^{1/2}\log\log{q}.$$

Although making a $o(1)$  improvement on the P\'{o}lya-Vinogradov inequality for all characters $\chi$ and intervals $(M,M+N]$ remains an open problem, there has been progress in determining general situations where such improvements can be made. Concerning short character sums, a classic result of Burgess~\cite{Bur,Bur1} states that for any primitive $\chi$
\begin{align*}
|S(\chi,M,N)|\le N^{1-1/r}q^{(r+1)/4r^2+o(1)},
\end{align*}
provided $r\le 3$ and for any $r\ge 2$ if $q$ is cubefree. Hildebrand~\cite{Hil} has shown that one can improve on the constant in the P\'{o}lya-Vinogradov inequality given estimates for short character sums and  Bober and  Goldmakher~\cite{BG1} and Fromm and Goldmakher~\cite{FG} have shown how improvements on the constant in the Polya-Vinogradov inequality may be used to obtain new estimates for short character sums. See also~\cite{KSY} for a logarithmic improvement on the Burgess bound for prime modulus.
\newline 

Concerning long character sums, Hildebrand~\cite{Hil1} has shown that if $\chi(-1)=1$ then
\begin{align*}
|S(\chi,\alpha q)|<\varepsilon q^{1/2}\log{q},
\end{align*}
for all $\alpha \in(0,1)$ except for a set of measure $q^{-c_1\varepsilon}$  and that if $\alpha=o(1)$ and $\chi(-1)=1$ then
\begin{align*}
S(\chi,\alpha q)=o(q^{1/2}\log{q}).
\end{align*}
 Bober and Goldmakher~\cite{BG} and  Bober,  Goldmakher,  Granville and  Koukoulopoulos~\cite{BGGK} have obtained much more precise results concerning the distribution of long character sums and  Granville and  Soundararajan~\cite{GS} have obtained results concerning the distribution of short character sums. Granville and Soundararajan~\cite{GS1} have also shown that 
$$S(\chi,N)\ll q^{1/2}(\log{q})^{1-\delta_g/2+o(1)},$$
if $\chi$ has odd order $g$, where 

$$\delta_g=1-\frac{g}{\pi}\sin{\frac{\pi}{g}},$$  and the factor $\delta_g/2$ occuring above has been improved by Goldmakher~\cite{Gold} and Lamzouri and Mangerel~\cite{LM}. 
\newline 

We consider the problem of estimating $S(\chi,M,N)$ uniformly over $\chi,M$ and $N$ in the P\'{o}lya-Vinogradov range. Since the work of P\'{o}lya and Vinogradov there have been a number of improvements to the constant $c$ occuring in~\eqref{eq:polyavinogradov}. The sharpest constant is due to  Pomerance~\cite{Pom} and is based on ideas of Landau~\cite{Lan} and an unpublished observation of Bateman, see~\cite{Hil}. In particular, Pomerance~\cite[Theorem~1]{Pom} shows that 
\begin{align*}
|S(\chi,M,N)|\le \begin{cases} \left(\frac{2}{\pi^2}+o(1) \right)q^{1/2}\log{q}, \quad \text{if} \quad \chi(-1)=1, \\  
\left(\frac{1}{2\pi}+o(1) \right)q^{1/2}\log{q}, \quad \text{if} \quad \chi(-1)=-1.
\end{cases}
\end{align*}

Pomerance gives the lower order terms explicitly and these have been improved by Frolenkov~\cite{Fro} and Frolenkov and Soundararajan~\cite{FS}. In the case of intervals starting from the origin one may obtain better constants with the sharpest given by  Granville and  Soundararajan~\cite{GS1}.
\newline

In this paper we obtain a new constant in the P\'{o}lya-Vinogradov inequality for arbitrary intervals. Our argument follows previously established techniques which use the Fourier expansion of an interval to reduce to Gauss sums. Our improvement comes from approximating an interval by a function with slower decay on the edges which allows for a better estimate of the $\ell_1$ norm of the Fourier transform. This induces an error for our original sums which we deal with by combining some ideas of Hildebrand~\cite{Hil} with Garaev and Karatsuba~\cite{GK}. A new feature of our argument is that we use estimates for long character sums to improve on the constant in the P\'{o}lya-Vinogradov inequality. For example, if one could show that for any $\varepsilon>0$ we have 
$$S(\chi,M,N)=o(q^{1/2}\log{q}),$$
for arbitrary $M$ whenever $N<q^{1-\varepsilon}$ and sufficiently large $q$  then it would follow from our argument that 
\begin{align*}
S(\chi,M,N)=o(q^{1/2}\log{q}),
\end{align*}
for arbitrary $M$ and $N$.
\section{Main result}
Our main result is as follows.
\begin{theorem}

\label{thm:main1}
For integer $q$ we define
\begin{align*}
c=\begin{cases} \frac{1}{4} \quad \text{if $q$ is cubefree}, \\ \frac{1}{3} \quad \text{otherwise}. \end{cases}
\end{align*}
For any primitive character $\chi \mod{q}$ and integers $M$ and $N$ we have 
\begin{align*}
\left|\sum_{M<n< M+N}\chi(n)\right|\le (1+o(1))\frac{4c}{\pi^2}q^{1/2}\log{q}.
\end{align*}
\end{theorem}
\section{Preliminary estimates  for character sums}
The aim of this section is to obtain estimates for long character sums which will be required for the proof of Theorem~\ref{thm:main1}.  The following Lemma is a consequence of the work of Burgess~\cite{Bur0,Bur,Bur1}.
\begin{lemma}
\label{lem:burgessmv}
Let $q,V$ and $r\ge 2$ be positive integers satisfying
\begin{align*}
V\le q^{1/2r},
\end{align*} 
and suppose $\chi$ is a primitive character mod $q$. Then we have 
\begin{align*}
\sum_{1\le v_1,\dots,v_{2r}\le V}\left|\sum_{\lambda=1}^{q}\chi\left(\frac{(\lambda+v_1)\dots (\lambda+v_r)}{(\lambda+v_{r+1})\dots (\lambda+v_{2r})}\right) \right|\ll q^{1/2+o(1)}V^{2r},
\end{align*}
for $r\le 3$ and any $r\ge 2$ provided $q$ is cubefree.
\end{lemma}

For a proof of the following, see~\cite{FI}.
\begin{lemma}
\label{lem:FI}
Let $q,M,N$ and $U$ be integers satisfying
$$2NU<q.$$
The number of solutions to the congruence 
$$n_1u_1\equiv n_2u_2 \mod{q},$$
with variables satisfying
$$M<n_1,n_2\le M+N, \quad 1\le u_1,u_2\le U,$$
is $O(NU\log{q})$.
\end{lemma}
The following is a variant of the Burgess bound for twists of characters to small modulus.
\begin{lemma}
\label{lem:burgess}
Let $q,M,N,k$ and $r$ be integers satisfying
$$N\le q^{1/2+1/4r}.$$
Let $\chi$ be a primitive character mod $q$ and $\psi$ be any multiplicative character mod $k$. Then we have 
\begin{align*}
\sum_{M<n\le M+N}\psi(n)\chi(n) \ll kN^{1-1/r}q^{(r+1)/4r^2+o(1)},
\end{align*}
for $r\le 3$ and any $r\ge 2$ provided $q$ is cubefree.
\end{lemma}
\begin{proof}
We fix an integer $r\ge 2$ and a sufficiently small $\varepsilon>0$ and proceed by induction on $N$. We formulate our induction hypothesis as follows. For any integers $M$ and $K$ we have 
\begin{align*}
\left|\sum_{M<n\le M+K}\psi(n)\chi(n)\right|\le ckK^{1-1/r}q^{(r+1)/4r^2+\varepsilon},
\end{align*}
for some constant $c$ to be determined later which may depend on $\varepsilon.$
Since the result is trivial for $K\le q^{1/4}$ this forms the basis of our induction. Define the integers 
$$U= \left \lfloor \frac{N}{8q^{1/2r}} \right\rfloor, \quad V=\left \lfloor \frac{q^{1/2r}}{k} \right\rfloor,$$
and note that 
\begin{align*}
UV\le \frac{N}{8k}.
\end{align*}
For any integer $y<N$ we have 
\begin{align*}
\sum_{M<n\le M+N}\psi(n)\chi(n)&=\sum_{M-y<n\le M+N-y}\psi(n+y)\chi(n+y) \\
&=\sum_{M<n\le M+N}\psi(n+y)\chi(n+y)+\sum_{M-y<n\le M}\psi(n+y)\chi(n+y) \\ &\quad \quad \quad -\sum_{M+N-y<n\le M+N}\psi(n+y)\chi(n+y),
\end{align*}
and hence by our induction hypothesis
\begin{align*}
\sum_{M<n\le M+N}\psi(n)\chi(n)=\sum_{M<n\le M+N}\psi(n+y)\chi(n+y)+ \frac{\theta c}{2}kN^{1-1/r}q^{(r+1)/4r^2+\varepsilon}.
\end{align*}
for some $|\theta|\le 1$  depending on $y$. Let $\cU$ denote the set 
$$\cU=\{ 1\le u \le U \ : \ (u,q)=1 \},$$
and average the above over integers of the form $kuv$ with $u\in \cU$ and $1\le v \le V$ to get 
\begin{align}
\label{eq:Wbb1}
\left|\sum_{M<n\le M+N}\psi(n)\chi(n)\right|\le \frac{1}{V|\cU|}|W|+ \frac{c}{2}kN^{1-1/r}q^{(r+1)/4r^2+\varepsilon},
\end{align}
where 
\begin{align*}
W=\sum_{M<n\le M+N}\sum_{u\in \cU}\sum_{1\le v \le V}\psi(n+kuv)\chi(n+kuv).
\end{align*}
Since $\psi$ has modulus $k$, we have 
\begin{align*}
|W|&\le \sum_{M<n\le M+N}\sum_{u\in \cU}\left|\sum_{1\le v \le V}\chi(nu^{-1}+kv) \right| \\
&=\sum_{\lambda=1}^{q}I(\lambda)\left|\sum_{1\le v \le V}\chi(\lambda+kv) \right|,
\end{align*}
where $I(\lambda)$ counts the number of solutions to the congruence
\begin{align*}
nu^{-1}\equiv \lambda \mod{q}, \quad M<n\le M+N, \ \ u\in \cU.
\end{align*}
By H\"{o}lder's inequality 
\begin{align*}
|W|^{2r}\le \left(\sum_{\lambda=1}^{q}I(\lambda) \right)^{2r-2}\left(\sum_{\lambda=1}^{q}I(\lambda)^2 \right)\left(\sum_{\lambda=1}^{q}\left|\sum_{1\le v \le V}\chi(\lambda+kv)\right|^{2r} \right).
\end{align*}
We have 
\begin{align*}
\sum_{\lambda=1}^{q}I(\lambda)=N|\cU|\le NU,
\end{align*}
and by Lemma~\ref{lem:FI}
\begin{align*}
\sum_{\lambda=1}^{q}I(\lambda)^2 \ll NU\log{q},
\end{align*}
since $NU\le q$. By Lemma~\ref{lem:burgessmv}
\begin{align*}
\sum_{\lambda=1}^{q}\left|\sum_{1\le v \le V}\chi(\lambda+kv)\right|^{2r} & \le \sum_{1\le v_1,\dots,v_{2r}\le V}\left|\sum_{\lambda=1}^{q}\chi\left(\frac{(\lambda+kv_1)\dots (\lambda+kv_r)}{(\lambda+kv_{r+1})\dots (\lambda+kv_{2r})}\right) \right| \\ 
&\le \sum_{1\le v_1,\dots,v_{2r}\le kV}\left|\sum_{\lambda=1}^{q}\chi\left(\frac{(\lambda+v_1)\dots (\lambda+v_r)}{(\lambda+v_{r+1})\dots (\lambda+v_{2r})}\right) \right| \\
&\le q^{1/2+o(1)}k^{2r}V^{2r},
\end{align*}
since $V\le q^{1/2r}/k,$ provided $r\le 3$ or $r\ge 2$ and $q$ cubefree. 
Combining the above estimate, we arrive at 
\begin{align*}
|W|^{2r}\ll k^{2r}(NU)^{2r-1}q^{1/2+o(1)}V^{2r},
\end{align*}
which after recalling the choice of $U$ and $V$ implies 
\begin{align*}
\frac{|W|}{|\cU|V}\ll kN^{1-1/2r}\frac{U^{1-1/2r}}{|\cU|}q^{1/4r+o(1)},
\end{align*}
and since  $|\cU|\ge Uq^{o(1)}$, we get 
\begin{align*}
\frac{|W|}{|\cU|V}\ll kN^{1-1/r}q^{(r+1)/4r^2+o(1)},
\end{align*}
and hence 
\begin{align*}
\frac{|W|}{|\cU|V}\le c_0kN^{1-1/r}q^{(r+1)/4r^2+\varepsilon/2},
\end{align*}
for some $c_0$ which may depend on $\varepsilon$, provided $q$ is sufficiently large.
Combining the above with~\eqref{eq:Wbb1} gives 
\begin{align*}
\left|\sum_{M<n\le M+N}\psi(n)\chi(n)\right|&\le c_0kN^{1-1/r}q^{(r+1)/4r^2+\varepsilon/2}+ \frac{c}{2}kN^{1-1/r}q^{(r+1)/4r^2+\varepsilon} \\ 
&\le c_0kN^{1-1/r}q^{(r+1)/4r^2+\varepsilon},
\end{align*}
on taking $c=c_0$ and assuming $q$ is sufficiently large.
\end{proof}
The following is due to Montgomery and Vaughan~\cite{MV}.
\begin{lemma}
\label{lem:mv}
Let $N$ be a positive integer and $\alpha$ a real number satisfying
$$\left|\alpha-\frac{a}{q}\right|\le \frac{1}{q^2},$$
for integers $a$ and $q$ satisfying $(a,q)=1$. Suppose that 
$$2\le R\le q \le \frac{N}{R},$$
for some positive number $R$. Then for any multiplicative function $f$ satisfying $|f|\le 1$ we have
\begin{align*}
\left|\sum_{1\le n \le N}f(n)e(\alpha n)\right|\ll \frac{N}{\log{N}}+\frac{N(\log{R})^{3/2}}{R^{1/2}}.
\end{align*}
\end{lemma}
The proof of the following estimate follows the proof of Hildebrand~\cite[Lemma~3]{Hil} and is based on Lemma~\ref{lem:burgess} and Lemma~\ref{lem:mv}.
\begin{lemma}
\label{lem:long111}
For integer  $q$  we define
\begin{align*}
c=\begin{cases} \frac{1}{4} \quad \text{if $q$ is cubefree}, \\ \frac{1}{3} \quad \text{otherwise}. \end{cases}
\end{align*}
For any primitive character $\chi$ mod $q$,  any $\varepsilon>0$, any real number $\alpha$ and any integer $N$ satisfying 
\begin{align}
\label{eq:long111Ncond1}
q^{c+\varepsilon}\le N \le q,
\end{align} we have
\begin{align*}
\left|\sum_{1\le n \le N}\chi(n)e(\alpha n) \right| \ll \frac{N}{\log{q}},
\end{align*}
provided $q$ is sufficiently large.
\end{lemma}
\begin{proof}
Let 
$$R=(\log{q})^3,$$
and apply Dirichlet's theorem to obtain integers $r$ and $k$ satisfying $1\le k \le N/R,$ $(r,k)=1$ and 
\begin{align}
\label{eq:alphaapprox}
\left|\alpha-\frac{r}{k}\right|\le \frac{R}{kN}.
\end{align}
If $k\ge R$ then by Lemma~\ref{lem:mv}
\begin{align*}
\left|\sum_{1\le n \le N}\chi(n)e(\alpha n)\right|\ll \frac{N}{\log{N}}+\frac{N(\log{R})^{3/2}}{R^{1/2}}\ll \frac{N}{\log{q}},
\end{align*}
and hence we may suppose $k\le R$. By~\eqref{eq:alphaapprox} and partial summation
\begin{align}
\label{eq:long111sum}
\nonumber \left|\sum_{1\le n \le N}\chi(n)e(\alpha n)\right|&\ll \left(1+N\left|\alpha-\frac{r}{k}\right|\right)\max_{M\le N}\left|\sum_{1\le n \le M}\chi(n)e_k(rn)\right| \\
&\ll (\log{q})^3\left|\sum_{1\le n \le M}\chi(n)e\left(\frac{rn}{k}\right)\right|,
\end{align}
for some $M\le N$. If $M\le q^{c+\varepsilon/2}$ then we bound summation over $m$ trivially to get 
\begin{align*}
\left|\sum_{1\le n \le N}\chi(n)e(\alpha n)\right|\ll q^{c+\varepsilon/2}\le \frac{N}{\log{q}},
\end{align*}
by~\eqref{eq:long111Ncond1}. Suppose next that 
\begin{align}
\label{eq:Mcase21111111111}
M\ge q^{c+\varepsilon/2}.
\end{align}
 We have 
\begin{align}
\label{eq:Sareduction}
\sum_{1\le n \le M}\chi(n)e\left(\frac{rn}{k}\right)=\sum_{a=1}^{k}e\left(\frac{ar}{k}\right)S(a),
\end{align}
where 
\begin{align*}
S(a)=\sum_{\substack{1\le n \le M \\ n\equiv a \mod{k}}}\chi(n).
\end{align*}
Fix some $1\le a \le k$ and consider $S(a)$. Let $d=(a,k)$ and write 
$$a'=\frac{a}{d}, \quad k'=\frac{k}{d},$$
so that 
\begin{align*}
S(a)&=\sum_{\substack{1\le n \le M \\ n\equiv a' \mod{k'} \\ n\equiv 0 \mod{d}}}\chi(n)=\frac{1}{\phi(k')}\sum_{\psi \mod{k'}}\overline \psi(a')\sum_{\substack{1\le n \le M  \\ n\equiv 0 \mod{d}}}\psi(n)\chi(n).
\end{align*}
If $(d,q)\neq 1$ then $S(a)=0$. If $(d,q)=1$ then by~\eqref{eq:Mcase21111111111} and Lemma~\ref{lem:burgess} we have 
\begin{align*}
|S(a)|\le \frac{1}{\phi(k')}\sum_{\psi \mod{k'}}\left|\sum_{\substack{1\le n \le M/d}}\psi(n)\chi(n)\right|\ll kNq^{-\delta}
\end{align*}
for some $\delta>0$ depending on $\varepsilon$. By~\eqref{eq:long111sum} and~\eqref{eq:Sareduction} this gives
\begin{align*}
\left|\sum_{1\le n \le N}\chi(n)e(\alpha n)\right|\ll k^2(\log{q})^3Nq^{-\delta}\ll (\log{q})^{9}Nq^{-\delta}\ll \frac{N}{\log{q}},
\end{align*}
which completes the proof.
\end{proof}
The proof of the following is based on some ideas of Garaev and Karatsuba~\cite{GK}.
\begin{lemma}
\label{lem:longcharacter}
For integer $q$ we define
\begin{align*}
c=\begin{cases} \frac{1}{4} \quad \text{if $q$ is cubefree}, \\ \frac{1}{3} \quad \text{otherwise}. \end{cases}
\end{align*}
For any primitive character $\chi$ mod $q$, any $\varepsilon>0$ and integers $M$ and $N$ with $N<q^{1-c-\varepsilon}$ we have 
\begin{align*}
\sum_{M<n\le M+N}\chi(n)\ll q^{1/2}.
\end{align*}

\end{lemma}
\begin{proof}
Expanding into Gauss sums, we have 
\begin{align*}
\left|\sum_{M<n\le M+N}\chi(n)\right|&=\frac{1}{q^{1/2}}\left|\sum_{M<n\le M+N}\sum_{0<|m|\le (q-1)/2}\overline \chi(m)e_q(mn) \right| \\
&=\frac{1}{q^{1/2}}\left|\sum_{0<n\le N}\sum_{0<|m|\le (q-1)/2}\overline \chi(m)e_q(m(M+n)) \right|.
\end{align*}
Let 
\begin{align*}
S_1=\sum_{0<n\le N}\sum_{0<|m|\le q/N}\overline \chi(m)e_q(m(M+n)),
\end{align*}
and 
\begin{align*}
S_2=\sum_{0<n\le N}\sum_{q/N<|m|\le (q-1)/2}\overline \chi(m)e_q(m(M+n)),
\end{align*}
so that 
\begin{align*}
\left|\sum_{M<n\le M+N}\chi(n)\right|\le \frac{1}{q^{1/2}}\left(|S_1|+|S_2| \right).
\end{align*}
Bounding $S_1$ trivially gives
\begin{align*}
|S_1|\le q,
\end{align*}
and hence
\begin{align}
\label{eq:S2}
\left|\sum_{M<n\le M+N}\chi(n)\right|\le q^{1/2}+\frac{|S_2|}{q^{1/2}}.
\end{align}
Considering $S_2$, we have 
\begin{align}
\label{eq:S222}
\nonumber S_2 &=\sum_{q/N<|m|\le (q-1)/2}\overline \chi(m)e_q(Mm)\frac{e_q((N+1)m)-e_q(m)}{e_q(m)-1} \\
&=S_{2,1}+S_{2,2},
\end{align}
where
\begin{align*}
S_{2,1}=\sum_{q/N<|m|\le (q-1)/2}\rho(m)\overline \chi(m)e_q((M+N+1)m),
\end{align*}
and
\begin{align*}
S_{2,2}=\sum_{q/N<|m|\le (q-1)/2}\rho(m)\overline \chi(m)e_q((M+1)m),
\end{align*}
and $\rho(m)$ is given by
$$\rho(m)=\frac{1}{e_q(m)-1}.$$
Considering $S_{2,1}$, by partial summation
\begin{align*}
S_{2,1}&=\sum_{q/N<t\le (q-1)/2}(\rho(t)-\rho(t+1))T(t)+\rho((q-1)/2)T((q-1)/2),
\end{align*}
where
$$T(x)=\sum_{q/N<|m|\le t}\overline \chi(m)e_q((M+N+1)m).$$
Since 
$$T((q-1)/2)\ll q^{1/2}\log{q},$$
and 
\begin{align*}
(\rho(t)-\rho(t+1))\ll \frac{q}{t^2},
\end{align*}
we have 
\begin{align*}
S_{2,1}\ll q^{1/2}\log{q}+q\sum_{q/N<t\le (q-1)/2}\frac{T(t)}{t^2}.
\end{align*}
Let 
$$T_0(t)=\sum_{0<|m|\le t}\overline \chi(m)e_q((M+N+1)m),$$
so that 
$$T_0(t)=T(x)+O\left(\frac{q}{N}\right),$$
and hence 
\begin{align*}
S_{2,1}&\ll q\sum_{q/N<t\le (q-1)/2}\frac{T_0(t)}{t^2}+q^{1/2}\log{q}+\frac{q^2}{N}\sum_{q/N<t<(q-1)/2}\frac{1}{t^2} \\
& \ll q\sum_{q/N<t\le (q-1)/2}\frac{T_0(t)}{t^2}+q.
\end{align*}
Since $N\le q^{1-c-\varepsilon}$ we have $q/N>q^{c+\varepsilon}$ and hence by Lemma~\ref{lem:long111}
\begin{align*}
S_{2,1}\ll \frac{q}{\log{q}}\sum_{q/N<t\le (q-1)/2}\frac{1}{t}+q\ll q.
\end{align*}
A similar argument shows that 
\begin{align*}
S_{2,2}\ll q,
\end{align*}
and hence by~\eqref{eq:S2} and~\eqref{eq:S222}
\begin{align*}
\left|\sum_{M<n\le M+N}\chi(n)\right|\ll q^{1/2},
\end{align*}
which completes the proof.
\end{proof}
\section{Estimate for the $\ell_1$ norm of an exponential sum}
In this section we estimate the $\ell_1$ norm of the Fourier transform of an approximation to an interval. The following is~\cite[Lemma~3]{Pom}
\begin{lemma}
\label{lem:pom}
For any real number $x$ and positive integer $n$ we have
$$\sum_{j=1}^{n}\frac{|\sin{jx}|}{j}\le \frac{2}{\pi}\log{n}+O(1).$$
\end{lemma}
\begin{lemma}
\label{lem:ff}
For integers $M,N$ and $K$ satisfying 
$$N+2K<q, \quad K\le q^{1-c},$$
for some $0<c<1$ we  define the function $f$ by
\begin{align*}
f(x)= 1 \quad \text{if} \quad M+1\le x \le M+N-1,
\end{align*}
\begin{align*}
f(x)=\frac{x}{K}+1-\frac{M+1}{K} \quad \text{if} \quad M+1-K\le x \le M+1,
\end{align*}
\begin{align*}
f(x)=-\frac{x}{K}+1+\frac{M+N-1}{K} \quad \text{if} \quad M+N-1\le x \le M+N-1+K,
 \end{align*}
\begin{align*}
f(x)=0 \quad \text{otherwise},
\end{align*}
and let $\widehat f(a)$ denote the Fourier transform of $f$
$$\widehat f(a)=\sum_{x=1}^{q}f(x)e_q(ax).$$
We have 
$$\sum_{a=1}^{q}|\widehat f(a)|\le (1+o(1))\frac{4q}{\pi^2}\log{(q/K)}. $$ 
\end{lemma}
\begin{proof}
For  $a\not \equiv 0 \mod{q}$ we have 
\begin{align*}
\widehat f(a)=S_1+S_2+S_3,
\end{align*}
where 
\begin{align*}
S_1=\sum_{M+1\le x \le M+N-1}e_q(ax),
\end{align*}
\begin{align*}
S_2=\sum_{M+1-K\le x \le M+1}\frac{x-M-1+K}{K}e_q(ax),
\end{align*}
and
\begin{align*}
S_3=\sum_{M+N-1\le x \le M+N-1+K}\frac{M+N-1+K-x}{K}e_q(ax).
\end{align*}
We have
\begin{align*}
S_1=e_q((M+1)a)\frac{1-e_q(a(N-1))}{1-e_q(a)},
\end{align*}
\begin{align*}
S_2 &=e_q((M+1)a)\frac{1}{K}\sum_{0\le x \le K}(K-x)e_q(-ax) \\ 
&=-\frac{e_q((M+2)a)}{1-e_q(a)}-\frac{e_q(Ma)(1-e_q(-Ka))}{K(1-e_q(-a))^2},
\end{align*}
and
\begin{align*}
S_3&=e_q((M+N-1)a)\frac{1}{K}\sum_{0\le x \le K}(K-x)e_q(ax) \\ 
&=\frac{e_q((M+N-1)a)}{1-e_q(a)}-\frac{e_q((M+N)a)(1-e_q(Ka))}{K(1-e_q(a))^2}.
\end{align*}
This implies that
\begin{align*}
\widehat f(a)=e_q((M-K+2)a)\frac{(1-e_q((N+K-2)a))(1-e_q(Ka))}{K(1-e_q(a))^2}+O(1),
\end{align*}
and hence 
\begin{align*}
|\widehat f(a)|\le \frac{|\sin{(\pi(N+K-2)a/q)}||\sin{(\pi Ka/q)}|}{K|\sin(\pi a/q)|^2}+O(1).
\end{align*}
Summing over $a=1,\dots,q$ gives 
\begin{align}
\label{eq:fT12}
\sum_{y=1}^{q}|\widehat f(a)|&\le 2\sum_{1\le a \le (q-1)/2}\frac{|\sin{(\pi(N+K-2)a/q)}||\sin{(\pi Ka/q)}|}{K|\sin(\pi a/q)|^2}+O(q) \\ 
&=2T_1+2T_2+O(q),
\end{align}
where 
\begin{align*}
T_1=\sum_{1\le a \le q/4K}\frac{|\sin{(\pi(N+K-2)a/q)}||\sin{(\pi Ka/q)}|}{K|\sin(\pi a/q)|^2},
\end{align*}
and 
\begin{align*}
T_2=\sum_{q/4K\le a \le (q-1)/2}\frac{|\sin{(\pi(N+K-2)a/q)}||\sin{(\pi Ka/q)}|}{K|\sin(\pi a/q)|^2}.
\end{align*}
We have 
\begin{align*}
T_1=\frac{q}{\pi}\sum_{1\le a \le q/4K}\left(\frac{|\sin{(\pi Ka/q)}|}{\pi K a/q}\right)\left(\frac{\pi a /q}{|\sin(\pi a/q)|} \right)^2\frac{|\sin{(\pi(N+K-2)a/q)}|}{a},
\end{align*}
and since $K\le q^{1-c}$ we get
\begin{align*}
T_1\le (1+o(1))\frac{q}{\pi}\sum_{1\le a \le q/4K}\frac{|\sin{(\pi(N+K-2)a/q)}|}{a},
\end{align*}
 hence by Lemma~\ref{lem:pom}
\begin{align*}
T_1\le (1+o(1))\frac{2q}{\pi^2}\log{(q/K)}.
\end{align*}
Considering $T_2$, we have 
\begin{align*}
T_2\ll \frac{q^2}{K}\sum_{q/4K\le a \le (q-1)/2}\frac{1}{a^2}\ll q.
\end{align*}
Combining the above with~\eqref{eq:fT12} gives
\begin{align*}
\sum_{y=1}^{q}|\widehat f(a)|\le (1+o(1))\frac{4q}{\pi^2}\log{(q/K)},
\end{align*}
and completes the proof.
\end{proof}
\section{Proof of Theorem~\ref{thm:main1}}
Considering the sum
\begin{align}
\label{eq:Sdef}
S=\sum_{M<n< M+N}\chi(n),
\end{align}
since 
$$
\sum_{M<n\le M+q}\chi(n)=0,
$$
by modifying $M$ if necessary we may assume that
\begin{align}
\label{eq:Ncond}
N< \frac{q}{2}.
\end{align}
Define $c$ by 
$$c=\begin{cases}\frac{1}{4} \quad \text{if $q$ is cubefree}, \\ \frac{1}{3} \quad \text{otherwise,} \end{cases}$$
 and for a sufficiently small $\varepsilon$ we let 
\begin{align}
\label{eq:Kdef}
K=\lfloor q^{1-c-\varepsilon}\rfloor.
\end{align}
Define the function $f$ by
\begin{align*}
f(x)= 1 \quad \text{if} \quad M+1\le x \le M+N-1,
\end{align*}
\begin{align*}
f(x)=\frac{x}{K}+1-\frac{M+1}{K} \quad \text{if} \quad M+1-K\le x \le M+1,
\end{align*}
\begin{align*}
f(x)=-\frac{x}{K}+1+\frac{M+N-1}{K} \quad \text{if} \quad M+N-1\le x \le M+N-1+K,
 \end{align*}
\begin{align*}
f(x)=0 \quad \text{otherwise}.
\end{align*}
Considering~\eqref{eq:Sdef},  we have
\begin{align*}
S&=\sum_{n}f(n)\chi(n)  -\sum_{M+1-K\le n \le M+1}\left(\frac{n}{K}+1-\frac{M+1}{K}\right)\chi(n) \\ & \quad \quad -\sum_{M+N-1\le n \le M+N-1+K}\left(-\frac{x}{K}+1+\frac{M+N-1}{K} \right)\chi(n).
\end{align*}
By partial summation and Corollary~\ref{lem:longcharacter}
\begin{align*}
\sum_{M+1-K\le n \le M+1}\left(\frac{n}{K}+1-\frac{M+1}{K}\right)\chi(n)\ll q^{1/2},
\end{align*}
and
\begin{align*}
\sum_{M+N-1\le n \le M+N-1+K}\left(-\frac{x}{K}+1+\frac{M+N-1}{K} \right)\chi(n)\ll q^{1/2},
\end{align*}
so that 
\begin{align*}
S=\sum_{n}f(n)\chi(n)+O(q^{1/2}).
\end{align*}
Hence it is sufficient to show
\begin{align}
\label{eq:S'def}
\left|\sum_{n=1}^{q}f(n)\chi(n)\right|\le \left(\frac{4c}{\pi^2}+o(1) \right)q^{1/2}\log{q}.
\end{align}
Expanding $f$  into a Fourier series and using Lemma~\ref{lem:ff}, we get 
\begin{align*}
\left|\sum_{n=1}^{q}f(n)\chi(n)\right|&\le \frac{1}{q}\sum_{a=1}^{q}|\widehat f(a)|\left|\sum_{n=1}^{q}\chi(n)e_q(-an)\right| \\
&\le  (1+o(1))\frac{4q}{\pi^2}\log{(q/K)}q^{1/2} \\ &=(1+o(1))\frac{4c}{\pi^2}q^{1/2}\log{q},
\end{align*}
and completes the proof.

\end{document}